%% file: main.tex
\documentclass[runningheads]{llncs}
\usepackage[utf8]{inputenc}
\usepackage{latexsym}
\usepackage{cite}
\usepackage[ruled,vlined]{algorithm2e}
\usepackage{amstext}
\usepackage{amssymb}
\usepackage{amsmath}
\usepackage{bbm}
\usepackage{amscd}
\usepackage{color}
\usepackage{graphicx}
\usepackage[english]{babel}
\usepackage{tabularx}
\usepackage{booktabs}
\usepackage{mathtools}
\mathtoolsset{showonlyrefs=true}
\usepackage{verbatim}
\usepackage{enumerate}
\usepackage{xcolor}
\usepackage{subcaption}
\usepackage{enumitem}
\usepackage{url}
\usepackage{hyperref}
\usepackage{mathabx}
\usepackage{enumitem}
\usepackage{float}

\usepackage[T1]{fontenc}
\newcommand{\R}{\mathbb{R}}
\newcommand{\N}{\mathbb{N}}

\newcommand{\sign}{\operatornamewithlimits{sign}}
\newcommand{\argmin}{\operatornamewithlimits{argmin}}

\newcommand{\xk}{x^{k}}
\newcommand{\xkk}{x^{k+1}}
\newcommand{\JpX}{\mathbf{J}_\mathcal{X}^{p}} 
\newcommand{\JqY}{\mathbf{J}_\mathcal{Y}^{q}}

\newcommand{\JpXstar}{\mathbf{J}_\mathcal{X^*}^{{p}^*}}

\newcommand{\lpn}{\ell^{(p_n)}} 
\newcommand{\lp}{{L^p(\Omega)}} 
\newcommand{\lpvar}{{L^{p(\cdot)}(\Omega)}} 

\newcommand{\rhopn}{\rho_{(p_n)}}
\newcommand{\rhobarpn}{\bar{\rho}_{(p_n)}}

\newcommand{\jrhopn}{\mathbf{J}_{\rhopn}}
\newcommand{\jrhobarpn}{\mathbf{J}_{\rhobarpn}}

\newcommand{\red}[1]{\textcolor{black}{\text{#1}}}

\graphicspath{{Figures/Phantom15/}}

\usepackage{tikz}
\usetikzlibrary{spy,decorations.fractals,shapes.misc}
\newcommand{\showpic}[1]{%
\begin{tikzpicture}[spy using outlines={rectangle, magnification=2, width = 0.75cm, height = 0.75cm,
connect spies, red, thick}]%
\draw (0,0) node [anchor=north] {\includegraphics[height=0.22\textwidth]{#1}};%
\spy on (0.1cm,-1.4cm) in node at (0.62cm, -.58cm);
\end{tikzpicture}%
}

\begin{document}
\title{Stochastic gradient descent for linear inverse problems in variable exponent Lebesgue spaces 
}
\titlerunning{Stochastic gradient descent in variable exponent Lebesgue spaces}
%
%
%
%

\author{Marta Lazzaretti\inst{1,3} \and
Zeljko Kereta\inst{2} \and
Claudio Estatico\inst{1} \and 
Luca Calatroni\inst{3}}
\authorrunning{M. Lazzaretti, Z. Kereta, C. Estatico, L. Calatroni}
%
\institute{Dip. di Matematica, Universit\`a di Genova, Via Dodecaneso 35, 16146, Italy\\
\email{lazzaretti@dima.unige.it}, \email{estatico@dima.unige.it} \\
\and
Dept. of Computer Science, University College London, UK\\
\email{z.kereta@ucl.ac.uk} 
\and
CNRS, UCA, Inria, Laboratoire I3S, Sophia-Antipolis, 06903,  France \\
\email{calatroni@i3s.unice.fr}
}

\maketitle              

\begin{abstract}
We consider a stochastic gradient descent (SGD) algorithm for solving linear inverse problems (e.g., CT image reconstruction) in the Banach space framework of variable exponent Lebesgue spaces $\lpn(\mathbb{R})$. Such non-standard spaces have been recently proved to be the appropriate functional framework to enforce pixel-adaptive regularisation in signal and image processing applications. Compared to its use in Hilbert settings, however, the application of SGD in the Banach setting of $\lpn(\mathbb{R})$ is not straightforward, due, in particular to the lack of a closed-form expression and the non-separability property of the underlying norm.  In this manuscript, we show that SGD iterations can effectively be performed using the associated modular function. Numerical validation on both simulated and real CT data show significant improvements in comparison to   SGD solutions both in Hilbert and other Banach settings, in particular when non-Gaussian or mixed noise is observed in the data.


\end{abstract}

\vspace{-0.2cm}
\section{Introduction}
The literature on iterative regularisation methods for solving ill-posed linear inverse problems in finite/infinite-dimensional Hilbert or Banach settings is very vast, see, e.g.,  \cite{Engl2000,Schuster2012} for surveys. Given two normed vector spaces $(\mathcal{X},\|\cdot\|_{\mathcal{X}})$ and $(\mathcal{Y},\|\cdot\|_{\mathcal{Y}})$, we are interested in the inverse problem
\begin{equation}  \label{eq:inv_prob}
\text{find}\quad x\in\mathcal{X}\quad\text{s.t.}\quad \mathcal{Y}\ni y=Ax+\eta,
\end{equation}
where $A\in\mathcal{L}(\mathcal{X};\mathcal{Y})$ is a bounded linear operator, and $\eta\in\mathcal{Y}$  denotes the (additive) noise perturbation of magnitude $\|\eta\|_{\mathcal{Y}}\le\delta$, $ \delta>0$, corrupting the measurements. 
Due to the ill-posedness, the standard strategy for solving \eqref{eq:inv_prob} consists in computing  $x^\star\in\argmin_{x\in\mathcal{X}} ~\Psi(x)$, where the functional $\Psi:\mathcal{X}\to \mathbb{R}_+\cup\{+\infty\}$ quantifies the fidelity of a candidate reconstruction to the measurements, possibly combined with a penalty or regularisation term enforcing prior assumptions on the sought quantity $x\in\mathcal{X}$. 
A popular strategy for promoting implicit regularisation through algorithmic optimisation consists in designing iterative schemes solving instances of the minimisation problem 
$\argmin_{x\in\mathcal{X}} ~ \| Ax-y \|_{\mathcal{Y}}~$
or, more generally
\begin{equation}  \label{eq:min_prob_general}
\tag{P}
\argmin_{x\in\mathcal{X}}~f(x)\qquad \mbox{with} \quad f(x)=\tilde f (Ax-y),
\end{equation}
where, for $y\in\mathcal{Y}$, the function $f(\cdot)=\tilde f(A \cdot - y):\mathcal{X}\to\R_{\ge 0}$ measures the discrepancy between the model observation $Ax$ and $y$. The iterative scheme has to be endowed with a robust criterion for its early stopping in order to avoid that the computed reconstruction overfits the noise \cite{Natterer86}. In this context, the role of the parameter tuning the amount of regularisation is thus played by nothing but the number of performed iterations. One-step gradient descent algorithms, such as the (accelerated) Landweber or the Conjugate Gradient, represent the main class of optimisation methods for the resolution of \eqref{eq:min_prob_general}, see e.g. \cite{Neubauer1988,Eicke1992,Piana1997}. 
  
The most well-studied cases consider $\mathcal{X}$ and $\mathcal{Y}$ to be Hilbert spaces, e.g., $\mathcal{X}=\mathcal{Y}=\ell^2(\R)$. In this setting, problem \eqref{eq:min_prob_general} takes the form
$
    \argmin_{x\in \ell^2(\R)} ~ \frac{1}{2}\|Ax-y\|^2_{\ell^2(\mathbb{R})}
$
and can be solved by a standard Landweber iterative scheme 
\begin{equation}  \label{eq:landweber}
\xkk=\xk-\mu_{k+1}A^*(A\xk-y),
\end{equation}
for $k\geq 0$, where $\mu_{k+1}>0$ denotes the algorithmic step-sizes. 
%
However, many inverse problems require a more complex setting to retrieve solutions with specific features, such as sharp edges, piecewise constancy, sparsity patterns and/or to model non-standard (e.g., mixed) noise in the data. Either $\mathcal{X}$ or $\mathcal{Y}$, or both, can thus be modelled as more general Banach spaces. Notable examples are standard Lebesgue spaces $\lp$ and, in discrete settings, sequence spaces $\ell^p(\mathbb{R})$ with $p\in [1,+\infty]\setminus \left\{2\right\}$. While the solution space $\mathcal{X}$ affects the choice of the specific iterative scheme to be used, the measurement (or data) space $\mathcal{Y}$ is naturally connected to the norm appearing in \eqref{eq:min_prob_general}.
For example, for Hilbert $\mathcal{X}=\ell^2(\mathbb{R})$ and Banach $\mathcal{Y}=\ell^p(\mathbb{R})$, an instance of \eqref{eq:min_prob_general} reads as
\begin{equation}   \label{eq:prob_banach}
    \argmin_{x\in\ell^2(\mathbb{R})} ~ \frac{1}{q}\|Ax-y\|^q_{\ell^p},\quad \mbox{with} \;\, q>1,
\end{equation}
for which a gradient descent-type scheme can still be used in the form  $\xkk=\xk- A^*\mathbf{J}_{\ell^p}^q(A\xk-y)$, where $\mathbf{J}_{\ell^p}^q:\ell^p(\mathbb{R}) \to \ell^{p^*}(\mathbb{R}) $ is the so-called $q$-duality map of $\ell^p(\mathbb{R})$, defined as $\mathbf{J}_{\ell^p}^q(\cdot)
= \partial\left(\frac{1}{{q}}\|\cdot\|_{\ell^p(\R)}^{{q}}\right)$. 
When both $\mathcal{X}$ and $\mathcal{Y}$ are Banach spaces, a popular algorithm for solving 
$$\argmin_{x\in\mathcal{X}} ~ \frac{1}{q}\|Ax-y\|^q_{\mathcal{Y}},\quad \mbox{with} \;\, q>1$$
is the dual Landweber method \cite{Schopfer2006}
\begin{equation}\label{eq:gd_banach}
\xkk=\JpXstar \left( \JpX(\xk)-\mu_{k+1}A^*\JqY(A\xk-y) \right)
\end{equation} 
where 
$\JpX: \mathcal{X} \to {\mathcal{X}^*}$, 
is the ${p}$-duality map of $\mathcal{X}$, $\mathbf{J}_\mathcal{X^*}^{p^*}: {\mathcal{X}^*} \to \mathcal{X}\,$ is its inverse with $p^*$ denoting the conjugate exponent of $p$, i.e.~$1/p + 1/p^* =1$. For other references of gradient-descent-type solvers in Banach settings, see, e.g.~\cite{Schopfer2006,Schuster2012,Jin_2012}.
%
%

A non-standard Banach framework for solving linear inverse problems is the one of variable exponent Lebesgue spaces $\lpvar$ and $\lpn(\mathbb{R})$ \cite{LpvarBOOK}.
These Banach spaces are defined in terms of a Lebesgue measurable function $p(\cdot):\Omega\rightarrow[1,+\infty]$, or a real sequence $(p_n)_n$, respectively, that assigns coordinate-wise exponents to all points in the domain.
Variable exponent Lebesgue spaces have proven useful in the design of adaptive regularisation, 
suited 
to model heterogeneous data  and complex noise settings. Iterative regularisation procedures in this setting have been recently studied \cite{Bonino23} and also extended to composite optimisation problems involving non-smooth penalty terms \cite{Lazzaretti_SISC22}.

While benefiting from several convergence properties, the use of such (deterministic) iterative algorithms may be prohibitively expensive in large-size applications as they require the use of all data at each iteration. In this work, we follow the strategy performed by the seminal work of Robbins and Monro \cite{Robbins1951} and adapt a stochastic gradient descent (SGD) strategy to the non-standard setting of variable exponent Lebesgue space, in order to reduce the per-iteration complexity costs. Roughly speaking, this is done by defining a suitable decomposition of the original problem and implementing an iterative scheme where only a batch of data, typically one, is used to compute the current update. Note that the use of SGD schemes has recently attracted the attention of the mathematical imaging community \cite{Jin2022,Kereta2023} due to its applicability in large-scale applications such as medical imaging \cite{Herman1993,Needell2015,Twyman2023}. However, its extension to variable exponent Lebesgue setting is not trivial due to some structural difficulties (e.g., non-separability of the norm), making the adaptation a challenging task.

\vspace{-0.2cm}
\paragraph{Contribution.} 
We consider an SGD-based iterative regularisation strategy for solving linear inverse problems in the non-standard Banach setting of variable exponent Lebesgue space $\lpn(\mathbb{R})$. To overcome the non-separability of the norm in such space, we consider updates defined in terms of a separable function, the modular function. Numerical investigation of the methodology on CT image reconstruction are reported to show the advantages of considering such non-standard Banach setting in comparison to standard Hilbert scenarios. Comparisons between the modular-based deterministic and stochastic algorithms confirm improvements of the latter w.r.t. CPU times.

\vspace{-0.2cm}
\section{Optimisation in Banach spaces}
\vspace{-0.15cm}
In this section we revise the main definitions and tools useful for solving a general instance of \eqref{eq:min_prob_general} in the general context of Banach spaces $\mathcal{X}$ and $\mathcal{Y}$.
For a real Banach space $(\mathcal{X},\|\cdot\|_{\mathcal{X}})$, we denote by $(\mathcal{X}^*, \|\cdot\|_\mathcal{X^*})$ its dual space and, for any $x\in\mathcal{X}$ and $x^*\in\mathcal{X}^*$,  by $\langle x^*,x\rangle=x^*(x)\in\R$ its duality pairing.

The following definition is crucial for the development of algorithms solving \eqref{eq:min_prob_general} in Banach spaces.  We recall that in Hilbert settings $\mathcal{H}\cong\mathcal{H}^*$ holds by the Riesz representation theorem, with $\cong$ denoting an isometric isomorphism. Hence, for $x\in\mathcal{H}$, the element $\nabla f(x) \in \mathcal{H}^*$ can be implicitly identified with a unique element in $\mathcal{H}$ itself, up to
the canonical isometric isomorphism, so that the design of gradient-type schemes is significantly simplified, as in \eqref{eq:landweber}. Since the same identification does not hold, in general, for a Banach space $\mathcal{X}$, we recall the notion of duality maps, which properly associate an element of $\mathcal{X}$ with an element  (or a subset) of $\mathcal{X}^*$ \cite{Cioranescu1990}. 

\begin{definition}
\label{def:map_dual}
Let $\mathcal{X}$ be a Banach space and $p>1$. The duality map $\JpX$ with gauge function $t\mapsto t^{p-1}$ is the operator $\JpX:\mathcal{X}\to 2^{\mathcal{X}^*}$ such that, for all $x\in\mathcal{X}$,
\[
\JpX(x)=\big\{x^{*}\in \mathcal{X}^{*}:
\left<x^*,x \right>=\|x\|_\mathcal{X} \|x^{*}\|_{\mathcal{X}^{*}}, \, \|x^{*}\|_{\mathcal{X}^{*}}=\|x\|_\mathcal{X}^{p-1}\big\}.
\]
\end{definition}

Under suitable smoothness assumptions on $\mathcal{X}$ \cite{Schuster2012}, $\JpX(x)$ is single valued at all $x\in\mathcal{X}$.
For instance, for $\mathcal{X}=\ell^p(\mathbb{R})$, with $p>1$, 
all duality maps are single-valued. 
The following Theorem (see  \cite{Cioranescu1990}) provides an operative definition and a more intuitive interpretation of the duality maps.  
\begin{theorem}[Asplund's Theorem]
\label{Theo:Asplund}
The duality map $\JpX$ is the subdifferential of the convex functional $h:x \mapsto \frac{1}{p}\|x\|_\mathcal{X}^p$, that is, $\JpX(\cdot) = \partial (\frac{1}{p}\| \cdot \|_\mathcal{X}^p)$.
\end{theorem}
\red{The following result is needed for the invertibility of the duality map.}
\begin{proposition}\cite{Schuster2012}
    Under suitable smoothness and convexity conditions on $\mathcal{X}$ and for $p>1$, for all $x\in\mathcal{X}$ and all $x^*\in\mathcal{X}^*$, there holds
    \begin{equation}   \label{pro:inverse_dual_map}
        \JpXstar(\JpX(x))=x\,, \quad \quad\JpX(\JpXstar(x^*))=x^*.
    \end{equation}
\end{proposition}
We notice that, if the gradient term $A^*\JqY(A\xk-y)$ vanishes in iteration \eqref{eq:gd_banach}, then $\xkk=\JpXstar(\JpX(\xk))=\xk$ by Proposition \ref{pro:inverse_dual_map}. 

For any $p,r>1$ and for any $x,h\in\ell^p(\mathbb{R})$, the explicit formula for $\mathbf{J}_{\ell^p}^{r}$ is 
\begin{equation}  \label{eq:duality_map_banach}
\langle\mathbf{J}_{\ell^p}^{r}(x),h\rangle=\|x\|_{p}^{r-p}\sum_{n\in\N} \sign(x_n)|x_n|^{p-1}h_n.
\end{equation}
Moreover, since $\left(\ell^p(\mathbb{R})\right)^* \cong \ell^{p^*}(\mathbb{R})$, then the inverse of the $r$-duality map $\mathbf{J}_{\ell^p}^{r}$ is nothing but $\left(\mathbf{J}_{\ell^p}^{r}\right)^{-1}=\mathbf{J}_{(\ell^p)^*}^{r^*}=\mathbf{J}_{\ell^{p^*}}^{r^*}$. 
Hence, the explicit analytical expression of its inverse $\left(\mathbf{J}_{\ell^p}^{r}\right)^{-1}=\mathbf{J}_{\ell^{p^*}}^{r^*}$ is also known \cite{Cioranescu1990}.
\vspace{-0.2cm}
\subsection{Variable exponent Lebesgue spaces $\lpn(\mathbb{R})$}
\vspace{-0.15cm}
 In the following, we will introduce the main concepts and definitions on the variable exponent Lebesgue spaces in the discrete setting of $\lpn(\mathbb{R})$. For surveys, we refer the reader to \cite{LpvarBOOK,CruzUribeFiorenzaBOOK}.  We define a family $\mathcal{P}$ of variable exponents as
\[
\footnotesize{\mathcal{P} :=\left\{ (p_n)_{n\in\N} \subset \R :   1 < p_-:=\inf_{n\in\N} p_n \leq p_+:=\sup_{n\in\N}p_n<+\infty\right\}. }
\]
\vspace{-0.2cm}
\begin{definition}
For $(p_n)_{n\in\N}\in\mathcal{P}$ and any real sequence $x=(x_n)_{n \in \N}$, 
\vspace{-0.2cm}
\begin{equation}
\rhopn(x):=\sum_{n\in\N} \vert x_n\vert ^{p_n} \;\quad \mbox{and} \qquad \rhobarpn(x):=\sum_{n\in\N} \frac{1}{p_n}\vert x_n\vert ^{p_n}
\label{eq:modular_rho_bar}
\vspace{-0.1cm}
\end{equation}
are called modular functions associated with the exponent map $(p_n)_{n\in\mathbb{N}}$.
\label{def:modular}
\end{definition}
\vspace{-0.2cm}
\begin{definition}
\label{def:lux_norm_lp_discr}
The Banach space $\lpn(\mathbb{R})$ is the set of real sequences $x=(x_n)_{n\in\N}$ such that $\rhopn\left(\frac{x}{\lambda}\right)<1$ for some $\lambda>0$. For any $x=(x_n)_{n\in\N}\in \lpn(\mathbb{R})$, the (Luxemburg) norm on $\lpn(\mathbb{R})$ is defined as 
\begin{equation}  \label{eq:lux_norm}
    \|x\|_{\lpn}:=\inf \left\{\lambda>0:\ \rhopn\left(\frac{x}{\lambda}\right) \le 1\right\}.
\end{equation}
\end{definition}

We now report a result from \cite{Bonino23} where a characterisation of the duality map $ \mathbf{J}^r_{\lpn}$ is given, in relation with \eqref{eq:duality_map_banach}.
\begin{theorem}
Given $(p_n)_{n\in\N}\in\mathcal{P}$,  then 
for each $x=\left(x_n\right)_{n\in\N}\in\lpn(\mathbb{R})$ and for any $r>1$, the duality map $\mathbf{J}^r_{\lpn}(x):\lpn(\mathbb{R})\to(\lpn)^*(\mathbb{R})$ is the linear operator defined, for  all $h=(h_n)_{n\in\N}\in\lpn(\mathbb{R})$ by:
\begin{equation}
    \langle \mathbf{J}^r_{\lpn}(x), h\rangle = \frac{1}{\sum_{n\in\N}\frac{p_n \vert x_n\vert^{p_n}}{\|x\|_{\lpn}^{p_n}}} \sum_{n\in\N}\frac{p_n \sign(x_n)\vert x_n\vert^{p_n-1}}{\|x\|_{\lpn}^{p_n-r}}h_n.
    \label{eq:Jlpvar_succ}
\end{equation}
\end{theorem}

By \eqref{eq:lux_norm}, we note that $\|\cdot\|_{\lpn}$ is not separable as its computation requires the solution of a minimisation problem involving all elements $x_n$ and $p_n$ at the same time. As a consequence, the expression \eqref{eq:Jlpvar_succ} is not suited to be used in a computational optimisation framework. The following result from \cite{Lazzaretti_SISC22} provides more flexible expressions associated to the modular functions \eqref{eq:modular_rho_bar}.
\begin{proposition}
The functions $\rhopn$ and $\rhobarpn$ in \eqref{eq:modular_rho_bar} are Gateaux differentiable at any $x=\left(x_n\right)_{n\in\N}\in\lpn(\mathbb{R})$. For $h=(h_n)_{n\in\N}\in\lpn(\mathbb{R})$ their derivatives read
\begin{equation}
\footnotesize{
\langle\jrhopn(x),h\rangle=\sum_{n\in\N} p_n\sign(x_n)|x_n|^{p_n-1}h_n,\quad 
\langle\jrhobarpn(x),h\rangle=\sum_{n\in\N} \sign(x_n)|x_n|^{p_n-1}h_n.
\label{eq:jrhobarpn}
}
\end{equation}
\end{proposition}

Notice that, although $\jrhopn$ and $\jrhobarpn$ are formally not duality maps, we adopt the same notation for the sake of consistency with Asplund Theorem \ref{Theo:Asplund}.

\vspace{-0.2cm}
\section{Modular-based gradient descent in $\lpn(\mathbb{R})$}
\vspace{-0.15cm}
\label{sec3}
Given $(p_n)_{n\in\mathbb{N}}, (q_n)_{n\in\mathbb{N}}\in\mathcal{P}$, we now discuss how to implement a deterministic gradient-descent (GD) type algorithm for solving an instance of \eqref{eq:min_prob_general} with $\mathcal{X}=\lpn(\mathbb{R})$ and $\mathcal{Y}=\ell^{(q_n)}(\mathbb{R})$. Recalling \eqref{eq:gd_banach}, GD iterations in this setting require knowing the duality map $\mathbf{J}^r_{\lpn}$ and its inverse.
However, as shown in \cite[Corollary 3.2.14]{LpvarBOOK}, such an inverse does not directly relate to the point-wise conjugate exponents of $(p_n)_{n\in\mathbb{N}}$ as the isomorphism between $(\lpn)^*(\mathbb{R})$ and $\ell^{(p_n^*)}(\mathbb{R})$ -differing from the standard $\ell^p$ constant case- is not isometric. As discussed in \cite{Bonino23}, the approximation 
$\left(\mathbf{J}^r_{\lpn}\right)^{-1}=\mathbf{J}^{r^*}_{(\lpn)^*}\approx \mathbf{J}^{r^*}_{\ell^{(p_n^*)}}$
can be used as an inexact (but explicit) formula for computing the duality map of  $(\lpn)^*(\mathbb{R})$. Under this assumption, the dual Landweber method can thus be used to solve the minimisation problem $\argmin_{x\in \lpn(\mathbb{R})}~\frac{1}{q}\|Ax-y\|_{\ell^{(q_n)}}^q, \quad q>1.$ 
Note, however, that the computation of the duality map $\mathbf{J}^p_{\lpn}$ requires the computation of $\|x\|_{\lpn}$ which, as previously discussed, makes the iterative scheme rather inefficient in terms of computational time. We thus follow \cite{Lazzaretti_SISC22} and define in Algorithm \ref{alg_modular_gd} a more efficient modular-based gradient descent iteration for the resolution of \eqref{eq:min_prob_general} in the general setting of variable exponent Lebesgue spaces.  The following set of assumptions needs to hold:
\begin{enumerate}[label=\textbf{A.\arabic*}]
    \item \label{assump1} $\nabla f : \lpn(\mathbb{R}) \to(\lpn)^*(\mathbb{R})$ is $(\textit{\texttt{p}} -1)-$H{\"o}lder-continuous with exponent $1<\textit{\texttt{p}}\le2$ and constant $K>0$.
    \item \label{assump2} There exists $c>0$ such that, for all $u,v\in\lpn(\mathbb{R})\,$,
\small{
\[
\langle \jrhobarpn(u)-\jrhobarpn(v), u-v\rangle \ge c \max \left\{ \|u-v\|_{\lpn}^\texttt{p}, \|\jrhobarpn(u)-\jrhobarpn(v)\|_{(\lpn)^*}^{\texttt{p}^*}\right\} .
\label{eq:SGD_mod_based_it}
\]}
\end{enumerate}

The latter bound was previously used in \cite{GuanSong2015,Lazzaretti_SISC22}. It is a compatibility condition between the ambient space $\lpn(\mathbb{R})$ and the H{\"o}lder smoothness properties of the residual function to minimise to achieve algorithmic convergence. 

The minimisation of the specific function $f$ of \eqref{eq:min_prob_general} is achieved solving at each iteration \eqref{eq:GD_var} the following minimisation problem:  
\begin{equation*}
    \xkk = \argmin_{u\in\lpn(\mathbb{R})} \rhobarpn(u)-\langle \jrhobarpn(\xk), u\rangle + \mu_k \langle \nabla f(\xk),u\rangle.
\end{equation*}

\begin{algorithm}[t!]
\caption{Modular-based Gradient Descent in $\lpn(\mathbb{R})$}
\label{alg_modular_gd}
\textbf{Parameters:} $\{\mu_k\}_k$ s.t. $
0<\Bar{\mu}\le \mu_k \le \frac{\textit{\texttt{p}}c(1-\delta)}{K}$ with $0<\delta<1$, for all $k\geq 0.$\\
\textbf{Initialisation:} $x^0\in\lpn(\mathbb{R})$.\\
{
\textbf{repeat}
\begin{equation}
   \footnotesize{ \xkk = \vert \jrhobarpn(\xk)-\mu_k\nabla f(\xk)\vert ^{\frac{1}{p_n-1}}\sign{(\jrhobarpn(\xk)-\mu_k\nabla f(\xk))}}
    \label{eq:GD_var}
\end{equation}
\textbf{until} convergence
}
\end{algorithm}

The following proof shows that the functional $\jrhobarpn$ defined by \eqref{eq:jrhobarpn} is invertible and gives a point-wise characterisation of its inverse.
\begin{proposition}
The functional $\jrhobarpn$ in \eqref{eq:jrhobarpn} is invertible. For all $v\in(\lpn)^*(\mathbb{R})$, 
\begin{equation}
     (\jrhobarpn)^{-1}(v)=
    \Big(|v_n|^{\frac{1}{p_n-1}}\sign(v_n)\Big)_{n\in\N} \in \lpn(\mathbb{R}).
\end{equation}
\end{proposition}
\begin{proof}
    By straightforward componentwise computation, we have 
    \begin{align*}
        &|\jrhobarpn(v_n)|^{\frac{1}{p_n-1}}\sign(\jrhobarpn(v_n))=|\jrhobarpn (v_n)|^{\frac{1}{p_n-1}-1}\jrhobarpn (v_n)\\
        &=|\jrhobarpn (v_n)|^\frac{2-p_n}{p_n-1}\jrhobarpn (v_n)=| \, |v_n|^{p_n-1}\sign(v_n)|^\frac{2-p_n}{p_n-1}|v_n|^{p_n-1}\sign(v_n)=v_n \, .
    \end{align*}
\end{proof}
By the Proposition above, the update rule \eqref{eq:GD_var} of Algorithm \ref{alg_modular_gd}, can be rewritten as
\[
\xkk = (\jrhobarpn)^{-1}\Big( \jrhobarpn(\xk)-\mu_k\nabla f(\xk)\Big).
\]
As a consequence, whenever $\nabla f(x_k)=0$ at some $k\geq 0$, a stationary point $\xkk=(\jrhobarpn)^{-1}\Big( \jrhobarpn(\xk)\Big)=\xk$ is found, as expected.


The following convergence result is a special case of \cite[Proposition 3.4]{Lazzaretti_SISC22} providing an explicit convergence rate for the iterates of Algorithm \ref{alg_modular_gd}.

\begin{proposition}
\label{prop:convergencerate}
Let $x^* \in  \lpn(\mathbb{R})$ be a minimiser of $f$ and let $(\xk)_k $ be the sequence generated by Algorithm \ref{alg_modular_gd}. If $(\xk)$ is bounded, then:
\begin{equation}
\label{eq:conv_GS}
f(x^k) - f(x^*) \le  \frac{\eta}{k^{\textit{\texttt{p}}-1}}, 
\end{equation}
where $\texttt{p}>1$ is defined in assumption \ref{assump1} and $\eta=\eta(\bar{\mu}, \delta, p_{-}, x^0,x^*)$.
\end{proposition}

Note that when the measurement space $\mathcal{Y}$ is a variable exponent Lebesgue space $\ell^{(q_n)}(\mathbb{R})$, a more effective and consistent choice for the objective function is the modular of the discrepancy between the model observation and the data, i.e. $f(x)=\Bar{\rho}_{(q_n)}(Ax-y)$. In this way, the heavy computations of the $\|\cdot\|_{\ell^{(q_n)}}$ norm and of its gradient are not required, making the iteration scheme faster. 

\vspace{-0.2cm}
\section{Stochastic modular-based gradient-descent in $\lpn(\mathbb{R})$}
\vspace{-0.15cm}
The key challenge for the viability of many deterministic iterative methods for real-world image reconstruction problems is their scalability to data-size.
For example, the highest per-iteration cost in emission tomography lies in the application of the entire forward operator at each iteration, whereas each image domain datum in computed tomography often requires several gigabytes of storage space. The same could thus be a bottleneck in the application of Algorithm \ref{alg_modular_gd}.
The stochastic gradient descent (SGD) paradigm addresses this issue \cite{Robbins1951}.

We partition the forward operator $A$, and the forward model into a finite number of block-type operators $A_1,\ldots,A_{N_s}$, where $N_s\in\mathbb{N}$ is the number of subsets of data.
The same partition is applied to the observations. 
Classical examples of this methodology include Kaczmarz methods in CT \cite{Herman1993,Needell2015}.
The SGD version of the iteration \eqref{eq:gd_banach} in Banach spaces 
takes the form
\begin{equation}\label{eq:sgd_banach}
\xkk=\JpXstar \left( \JpX(\xk)-\mu_{k+1}A_{i_k}^*\JqY(A_{i_k}\xk-y) \right),
\end{equation}
where the indices $i_k\in\left\{1,\ldots,N_s\right\}$ are sampled uniformly at random.
Sampling reduces the per-iteration computational cost in $\mathcal{Y}$ by a factor of $N_s$.
In  \cite{Kereta2023} convergence of the iterates to a minimum norm solution is shown.
\begin{theorem}\label{thm:as_lin_convergence}
Let $\sum_{k=1}^\infty\mu_{k}=+\infty$  and $\sum_{k=1}^\infty \mu_{k}^{p^\ast} <+\infty.$
Then 
\begin{align*}
    \mathbb{P}\Big(\lim_{k\rightarrow\infty} \inf_{\widetilde x\in \mathcal{X}_{\min}}\|{\xkk-\widetilde x}\|_{\mathcal X}=0\Big)=1.
\end{align*}
Let $\JpX(x_0)\in\overline{{\rm range}(A^\ast)}$ and let
$\mu_{k}^{p^\ast-1}\leq \frac{C}{L_{\max}^{p^\ast}}$ for all $k\geq 0$ and some constant $C>0$, where $L_{\max}=\max_i\|A_i\|$.
Then $\lim_{k\rightarrow\infty} \mathbb{E}[\|\xkk-x^\dag\|_{\mathcal{X}}]=0$ 
 $\lim_{k\rightarrow\infty} \mathbb{E}[\|\JpX(\xkk)-\JpX(x^\dag)\|^{p^\ast}]=0$.
\end{theorem}

For noisy measurements, the regularising property of SGD  should be established by defining suitable stopping criteria.
However, robust stopping strategies are hard to use in practice and having methods that are less sensitive to data overfit is crucial for their practical use.  Note that \eqref{eq:sgd_banach} is the standard form of SGD for separable objectives. Namely, for $f(x)=\|Ax-y\|_q^q$, we can choose $f_i(x;A,y)=\|A_ix-y_i\|_q^q$, so that $f(x)=\sum_{i=1}^{N_s}f_i(x)$. By Theorem \ref{Theo:Asplund}, this decomposition shows that each step of \eqref{eq:sgd_banach} can thus be computed by simply taking a sub-differential of a single sum-function $f_i$.

To define a suitable SGD in variable exponent Lebesgue spaces, we  take as objective function $f(x)=\bar{\rho}_{(q_n)}(Ax-y)$ and split it into $N_s\ge1$ sub-objectives $f_i(x):=\bar{\rho}_{(q_n^i)}(A_ix-y_i)$, so that $\nabla f_i(x)=A_i^\ast\mathbf{J}_{\Bar{\rho}_{(q_n^i)}}(A_ix-y_i)$.  
\red{Exponents} $(q_n^i)_n$ are obtained through the same partition of the exponents $(q_n)_n$ as the one used to split up the data. Then, at iteration $k$ \red{and a} randomly sampled index $1\le i_k\le N_s$, the corresponding stochastic iterates are given by
\begin{equation*}
    \xkk = \argmin_{u\in\lpn(\mathbb{R})} ~\rhobarpn(u)-\langle \jrhobarpn(\xk), u\rangle + \mu_k \langle \nabla f_{i_k}(\xk),u\rangle.
\end{equation*}
The pseudocode of the resulting stochastic modular-based gradient descent in $\lpn(\R)$ is reported in Algorithm \ref{alg_modular_sgd}.
We expect that through minimal modifications an  analogous convegence result as Theorem \ref{thm:as_lin_convergence} can be proved in this setting too. A detailed convergence proof, however, is left for future research.

\begin{algorithm}[t!]
\caption{Stochastic Modular-based Gradient Descent in $\lpn(\mathbb{R})$}
\small{
\textbf{Parameters:} $\mu_0$ s.t. $
0<\Bar{\mu}\le \mu_0 \le \frac{\textit{\texttt{p}}c(1-\delta)}{K}$, $0<\delta<1$, $N_s\ge 1$, $\gamma>0$, $\eta>0$.\\
\textbf{Initialisation:}  $x^0\in\lpn(\mathbb{R})$.\\

\textbf{repeat}
\begin{itemize}
        \item[] Select uniformly at random $i_k\in\{1,\cdots,N_s\}$.
        \item[] Set
        $\mu_k=\frac{\mu_0}{1+\eta (k/N_s) ^ \gamma}$
        \item[] Compute
        \begin{equation}
        \footnotesize{
    \xkk = \vert \jrhobarpn(\xk)-\mu_k
    \nabla f_{i_k}(\xk)
    \vert ^{\frac{1}{p_n-1}}\sign{(\jrhobarpn(\xk)-\mu_k 
    \nabla f_{i_k}(\xk) )}
    }
    \label{eq:SGD_var}
\end{equation}
\end{itemize}
\textbf{until} convergence
}\label{alg_modular_sgd}
\end{algorithm}


\input{sec5.tex}

\vspace{-0.25cm}
\section{Conclusions}
\vspace{-0.15cm}
We proposed a stochastic gradient descent algorithm for solving linear inverse problems in 
 $\lpn(\mathbb{R})$. After recalling its deterministic counterpart and the difficulties encountered due to the non-separability of the underlying norm, a modular-based stochastic algorithm enjoying fast scalability properties is proposed. Numerical results show improved performance in comparison to standard $\ell^2(\mathbb{R})$ and $\ell^p(\mathbb{R})$-based algorithms and significant computational gains. Future work should adapt the convergence result (Theorem \ref{alg_modular_sgd}) to this setting and consider proximal extensions for incorporating non-smooth regularisation terms.

 \vspace{-0.25cm}
 {\small \section{Acknowledgements}
 \vspace{-0.15cm}
CE and ML acknowledge the support of the Italian INdAM group on scientific calculus GNCS. LC acknowledges the support received by the ANR projects TASKABILE (ANR-22-CE48-0010) and MICROBLIND (ANR-21-CE48-0008), the H2020 RISE projects NoMADS (GA. 777826) and the GdR ISIS project SPLIN. ZK acknowledges support from EPSRC grants EP/T000864/1 and EP/X010740/1.
}
\vspace{-0.2cm}
\bibliographystyle{plain}
\bibliography{ref.bib}
\end{document}

%% file: sec5.tex
\vspace{-0.2cm}
\section{Numerical results}
\vspace{-0.15cm}
We now present experimental results of the proposed Algorithm \ref{alg_modular_sgd} on two exemplar problems in computed tomography (CT). 
The first set of experiments consider a simulated setting for quantitatively comparing the performance of Algorithm 
\ref{alg_modular_sgd} with the corresponding Hilbert and Banach space versions \eqref{eq:sgd_banach}.
In the second set of experiments we consider the dataset of real-world CT scans of a walnut \red{taken from} \url{doi:10.5281/zenodo.4279549}, with a fan beam geometry.
For these data, we utilise the insights from the first set of experiments and apply Algorithm \ref{alg_modular_sgd} in a setting with different noise modalities across the sinogram space.
%
The experiments were conducted in \texttt{python}, using the open source  package \cite{CIL2021} for the tomographic backend.

\vspace{-0.25cm}

\paragraph{Hyper-parameter selection.} In the following experiments, we employ a decaying stepsize regime such that it satisfies the conditions of Theorem \ref{thm:as_lin_convergence} for the convergence of Banach space SGD, cf. \cite{Kereta2023}. A need for a decaying stepsize regime is common for stochastic gradient descent to mitigate the effects of inter-iterate variance.
Specifically, we use $\mu_k=\frac{\mu_{0}}{1+c(k/N_s)^\gamma}$, where $\mu_0>0$ is the initial stepsize, and $\gamma>0$ and $c>0$ control the decay speed. For the Hilbert space setting, $\mathbf{SGD}_2$, initial stepsize $\mu_0$ is given by the Lipschitz constant of the gradient of the objective function, namely $\mu_0=0.95/\max_{i}\|A_i\|^2$. For $\mathbf{SGD_p}$ and $\mathbf{SGD_{p_n,q_n}}$ the estimation of the respective H\"older continuity constant is more delicate and $\mu_0$ has to be tuned to guarantee convergence. However, its tuning is rather easy and the employ of a decaying strategy makes the choice of $\mu_0$ less critical.

As far as variable exponents are concerned, it is difficult (and somehow undesirable) to have a unified configuration as their selection is strictly problem-related. Parameters $(q_n)_n$ are related to the regularity of the measured sinograms as well as the different noise distributions considered.
For instance, when impulsive noise is considered, values of $q_-$ and $q_+$ closer to 1 are preferred while and for Gaussian noise values closer to $2$ are more effective. Solution space parameters $p_-$ and $p_+$ relate to the regularity of the solution to retrieve. As a consequence, their choice is intrinsically harder. We refer the reader to \cite{Bonino23}, where a comparison between different choices for $p_-$ and $p_+$ and different interpolation strategies is carried out for image deblurring with gradient descent \eqref{eq:gd_banach} in $\lpn$. 

\vspace{-0.25cm}
\paragraph{Simulated data.} We considered \eqref{eq:inv_prob} with $A$ given by the discrete Radon transform. For its definition we use a 2D parallel beam geometry, with 180 projection angles on a 1 angle separation, 256
detector elements, and pixel size of 0.1.
The synthetic phantom was provided by the CIL library, see Figure \ref{fig:phantom15}(b). 
After applying the forward operator, a high level (15\%) of salt-and-pepper noise is applied to the sinogram. The noisy sinogram is shown in Figure \ref{fig:phantom15}(a). 
\begin{figure}[t!]
\centering
\small
\setlength{\tabcolsep}{0pt}
\begin{tabular}{cccc}
{
\begin{tikzpicture}
\draw (0,0) node [anchor=south] {\includegraphics[height=0.22\textwidth]{{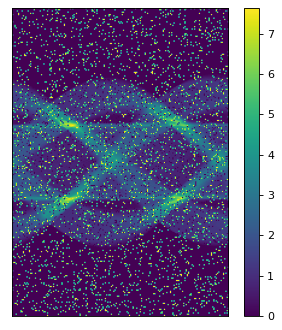}}};%
\end{tikzpicture}
\hspace*{-4mm}
}
&
\showpic{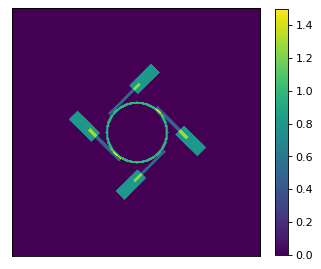}\hspace*{-4mm} &
{
\begin{tikzpicture}[spy using outlines={rectangle, magnification=2, width = 0.75cm, height = 0.75cm,
connect spies, red, thick}]%
\draw (0,0) node [anchor=north] {\includegraphics[height=0.22\textwidth]{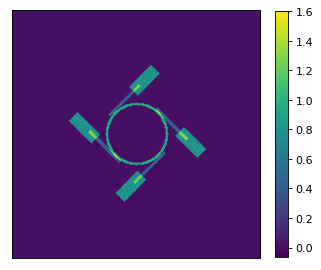}};%
\spy on (0.1cm,-1.4cm) in node at (0.61cm, -.6cm);
\end{tikzpicture}}
\hspace*{-4mm}
&
\begin{tikzpicture}[spy using outlines={rectangle, magnification=2, width = 0.75cm, height = 0.75cm,
connect spies, red, thick}]%
\draw (0,0) node [anchor=north] {\includegraphics[height=0.22\textwidth]{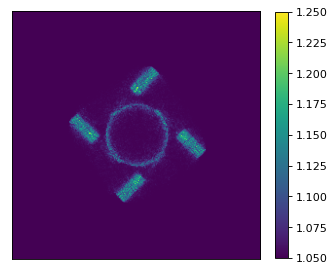} 
};%
\spy on (0.1cm,-1.4cm) in node at (0.54cm, -.61cm);
\end{tikzpicture}\hspace*{-4mm} \\
{\scriptsize{(a) Sinogram}} & {\scriptsize{(b) GT}} & {\scriptsize{(c) $\mathbf{SGD_{p_n,q_n}}$}} & {\scriptsize{(d) $(p_n)$ map}}
\end{tabular}
\caption{\footnotesize In (c) reconstruction of noisy sinogram (a) by $\mathbf{SGD_{p_n,q_n}}$, where $1.05= p_-\le(p_n)\le p_+=1.25$ is shown in (d) and $1.05=q_-\le(q_n)\le q_+=1.25$ is based on the model observation corresponding to $(p_n)$.}
\label{fig:phantom15}
\vspace{-0.25cm}
\end{figure}
\begin{figure}[t!]
\centering
\begin{subfigure}[b]{0.33\textwidth}
\includegraphics[height=3.5cm]{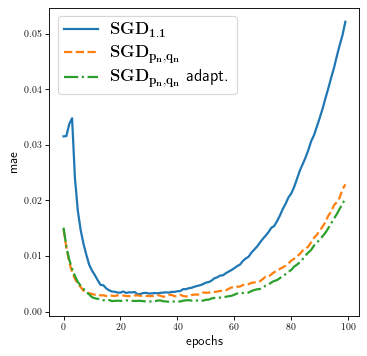}
\caption{MAE}
\label{fig:mae}
\end{subfigure}
\begin{subfigure}[b]{0.31\textwidth}
\includegraphics[height=3.5cm]{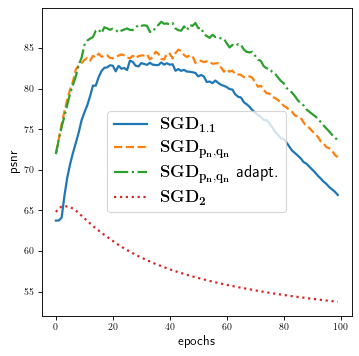}
\caption{PSNR}
\label{fig:psnr}
\end{subfigure}
\begin{subfigure}[b]{0.32\textwidth}
\includegraphics[height=3.5cm]{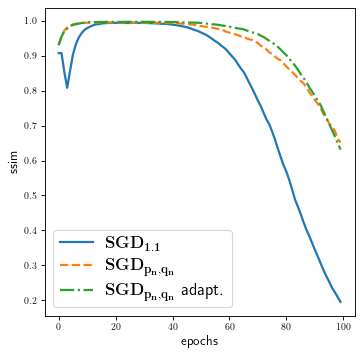}
\caption{SSIM}
\label{fig:ssim}
\end{subfigure}
\caption{\footnotesize Quality metrics along the first 100 epochs of $\mathbf{\text{SGD}_2}$; $\mathbf{\text{SGD}_{1.1}}$; $\mathbf{\text{SGD}_{p_n,q_n}}$ with and without adapting the exponent maps $(p_n)$. $\mathbf{\text{SGD}_2}$ is omitted from MAE and SSIM to improve the readability of the plots, due to its poor performance.}
\label{fig:phantom15_metrics}
\vspace{-0.15cm}
\end{figure}

\begin{table}[t]
\begin{center}
\begin{tabular}{c@{\hskip 0.2in}ccc@{\hskip 0.1in}cccccc}
\cmidrule[\heavyrulewidth]{2-10}
& \multicolumn{2}{c}{Deterministic} &      & \multicolumn{6}{c}{Stochastic ($\cdot=\mathbf{S}$)}                                            \\ \cmidrule{2-3} \cmidrule{5-10} 
Algorithm                   & It. & Tot. & & It.& Epoch & Tot.  & MAE & PSNR & SSIM \\ 
\cmidrule{1-3} \cmidrule{5-10} 
$\cdot\mathbf{GD_2}$                        & 0.44s         & 1324s & &0.02s         & 0.74s    & 74.4 s & 2.582{\rm e-}1   & 57.89  & 0.0304  \\
$\cdot\mathbf{GD_{1.1}}$                    & 0.43s         & 1297s & & 0.03s         & 0.81s    & 81.3s   & 3.671{\rm e-}3 & 82.64 & 0.9897   \\
$\cdot\mathbf{GD_{p_n,q_n}}$                     & 0.47s         & 1403s & & 0.03s         & 0.96s    & 96.5s  & 2.887{\rm e-}3 & 84.05 & 0.9927 \\
$\cdot\mathbf{GD_{p_n,q_n}}$ adapt.        & 0.44s         & 1317s & & 0.03s         & 0.91s    & 91.2s  &  1.777{\rm e-}3 & 88.10 & 0.9965  \\
Compute $(p_n),\ (q_n)$ & 0.45s         & 16s   & & 0.03s         & 0.8s      & 4.0s & - & - & -\\
\bottomrule
\end{tabular}
\end{center}
\caption{\footnotesize Comparison of per iteration cost and total CPU times after 3000 iterations for determistic algorithms and after 100 epochs for stochastic algorithm with $N_s=30$. MAE, PSNR and SSIM values for stochastic algorithms are computed after 40 epochs (before noise overfitting).
}
\label{tab:cputime}
\vspace{-0.7cm}
\end{table}
To compute subset data $A_i$ and $y_i$, the forward operator and the sinogram are pre-binned according to equally spaced views (w.r.t.~the number of subsets) of the scanner geometry. Subsequent subset data are offset from one another by the subset index $i$.
We consider $N_s=30$ batches. We compare results obtained by solving \eqref{eq:min_prob_general} by:
\begin{itemize}
    \item[] $\mathbf{{SGD}_{2}}$: $\mathcal{X}=\mathcal{Y}=\ell^2(\mathbb{R})$, $\mathcal{Y}=\ell^2(\mathbb{R})$, $f(x)=\frac{1}{2}\|Ax-y\|_2^2$ by SGD; 
    \item[] $\mathbf{{SGD}_p}$:  $\mathcal{X}=\mathcal{Y}=\ell^p(\mathbb{R})$, $p=1.1$, $f(x)=\frac{1}{p}\|Ax-y\|_p^p$ by Banach SGD  \eqref{eq:sgd_banach}; 
    \item[] $\mathbf{{SGD}_{p_n,q_n}}$: $\mathcal{X}=\lpn(\mathbb{R})$, $\mathcal{Y}=\ell^{(q_n)}(\R)$ for appropriately chosen exponent maps, $f(x)=\Bar{\rho}_{(q_n)}(Ax-y)$ with modular-based SGD Algorithm \ref{alg_modular_sgd}.
\end{itemize}
We considered step-sizes $\mu_k=\frac{\mu_{0}}{1+0.1(k/N_s)^\gamma}$, with $\mu_{0}$ and $\gamma$ which depend on the algorithm.\footnote{For $\mathbf{{SGD}_2}$ $\mu_0$ is set as $0.95/\max_{i}\|A_i\|^2$ and $\gamma=0.51$. For $\mathbf{{SGD}_p}$ and $\mathbf{{SGD}_{p_n,q_n}}$, we use $\mu_0=0.015$ with $\gamma=(p-1)/p+0.01$ and $\gamma=(p_--1)/p_-+0.01$ respectively.}
Spaces $\lpn(\mathbb{R})$ allow for variable exponent maps sensitive to local assumptions on both the solution and the measured data. 
A possible strategy for informed pixel-wise variable exponents consists in basing them on observed data (for $(q_n)$) and an approximation of the reconstruction (for $(p_n)$), as done in \cite{Bonino23,Lazzaretti_SISC22,Estatico2019}. 
To this end, we first compute an approximate reconstruction $\Tilde{x}\in\lpn(\mathbb{R})$ by running $\mathbf{{SGD}_p}$ in $\ell^{1.1}(\mathbb{R})$ for 5 epochs with a constant stepsize regime. 
The map $(p_n)$ is then computed via a linear interpolation of $\Tilde{x}$ between $p_{-}=1.05$ and $p_{+}=1.25$.  
The map $(q_n)$ is chosen as the linear interpolation between $q_{-}=1.05$ and $q_{+}=1.25$ of $A(p_n)$. 
The bounds $p_-,p_+$ and $q_-,q_+$ are chosen by prior assumptions on $y$ (sparse phantom) and on the noise observed (impulsive). 
%
We also tested an adaptive strategy by updating $(p_n)$ based on the current solution estimate once every $\beta_\text{updates}$ epochs to adapt the exponents along the iterations. 

In Figure \ref{fig:phantom15_metrics}, we report the mean absolute error (MAE), peak signal to noise ratio (PSNR) and structural similarity index (SSIM) of the iterates $\xk$ w.r.t.~the known ground-truth phantom along the first 100 epochs. 
Since PSNR favours smoothness, it is thus beneficial for $\mathbf{{SGD}_2}$, whereas MAE promotes sparsity hence is beneficial for both $\mathbf{{SGD}_p}$ and $\mathbf{{SGD}_{p_n,q_n}}$.
Figure \ref{fig:psnr} shows that  Banach space algorithms provide better performance than $\mathbf{{SGD}_2}$ in all three quality metrics. 
Note that all the results show the well-known semi-convergence behaviour with respect to the metrics considered.
To avoid such behaviour an explicit regulariser or a sound early stopping criterion would be beneficial for reconstruction performance. 
We observe that the use variable exponents does not only improve all quality metrics, but also makes the algorithm more stable: the quality of the reconstructed solutions is significantly less sensitive to the number of epochs,  making possible early stopping strategies more robust. 

In Table \ref{tab:cputime}, the CPU times for deterministic ($\mathbf{GD_2}$, $\mathbf{GD_p}$ and $\mathbf{GD_{p_n,q_n}}$) approaches and stochastic ones ($\mathbf{SGD_2}$, $\mathbf{SGD_p}$ and $\mathbf{SGD_{p_n,q_n}}$) are compared.

\vspace{-0.2cm}
\paragraph{Real CT datasets: walnut.}
 We consider a cone beam CT dataset of a walnut \cite{Meaney2020}, from which we take a 2D fan beam sinograms from the centre plane of the cone. The cone beam data uses $0.5$ angle separation over the range $[0,360]$. The used sinogram is obtained by pre-binning the raw data by a factor of $8$, resulting in $280$ effective detector pixels.
The measurements have been post-processed for dark current and flat-field compensation.
As stepsize we used $\mu_k=\frac{\mu_{0}}{1+0.001(k/N_s)^\gamma}$, with $N_s=10$ subsets, and suitable $\mu_{0}$ and $\gamma$. \footnote{For $\mathbf{SGD}_2$ , $\mu_0=0.95/\max_{i}\|A_i\|^2$, $\gamma=0.51$. For $\mathbf{SGD}_{p_n,q_n}$ we $\mu_0=0.001$, $\gamma=0.58$.} 
Initial images are computed by $5$ epochs of $\mathbf{SGD_{1.4}}$  with a constant stepsize.
\begin{figure}[t!]
\centering
\small
\setlength{\tabcolsep}{0pt}
\begin{tabular}{cccc}
\begin{tikzpicture}[spy using outlines={rectangle, magnification=3, width = 0.75cm, height = 0.75cm,
connect spies, red, thick}]%
\draw (0,0) node [anchor=north] {\includegraphics[height=0.24\textwidth]{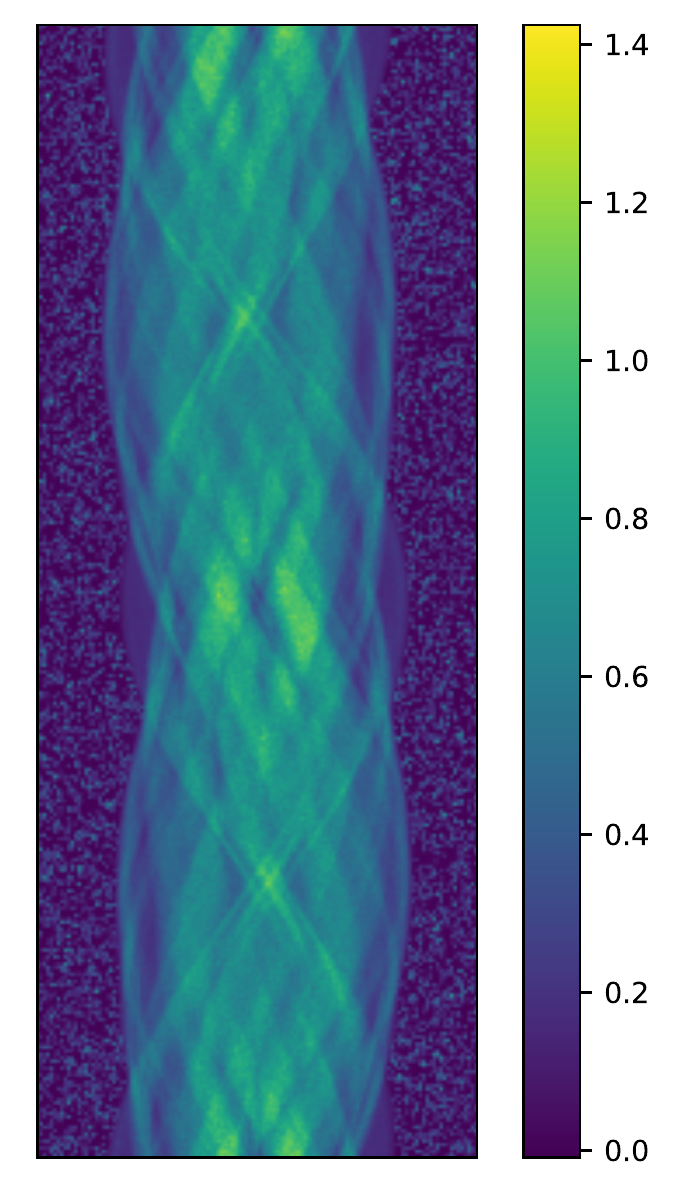} 
};%
\end{tikzpicture}\hspace*{-3mm}
&
\begin{tikzpicture}[spy using outlines={rectangle, magnification=3, width = 0.75cm, height = 0.75cm,
connect spies, red, thick}]%
\draw (0,0) node [anchor=north] {\includegraphics[height=0.24\textwidth]{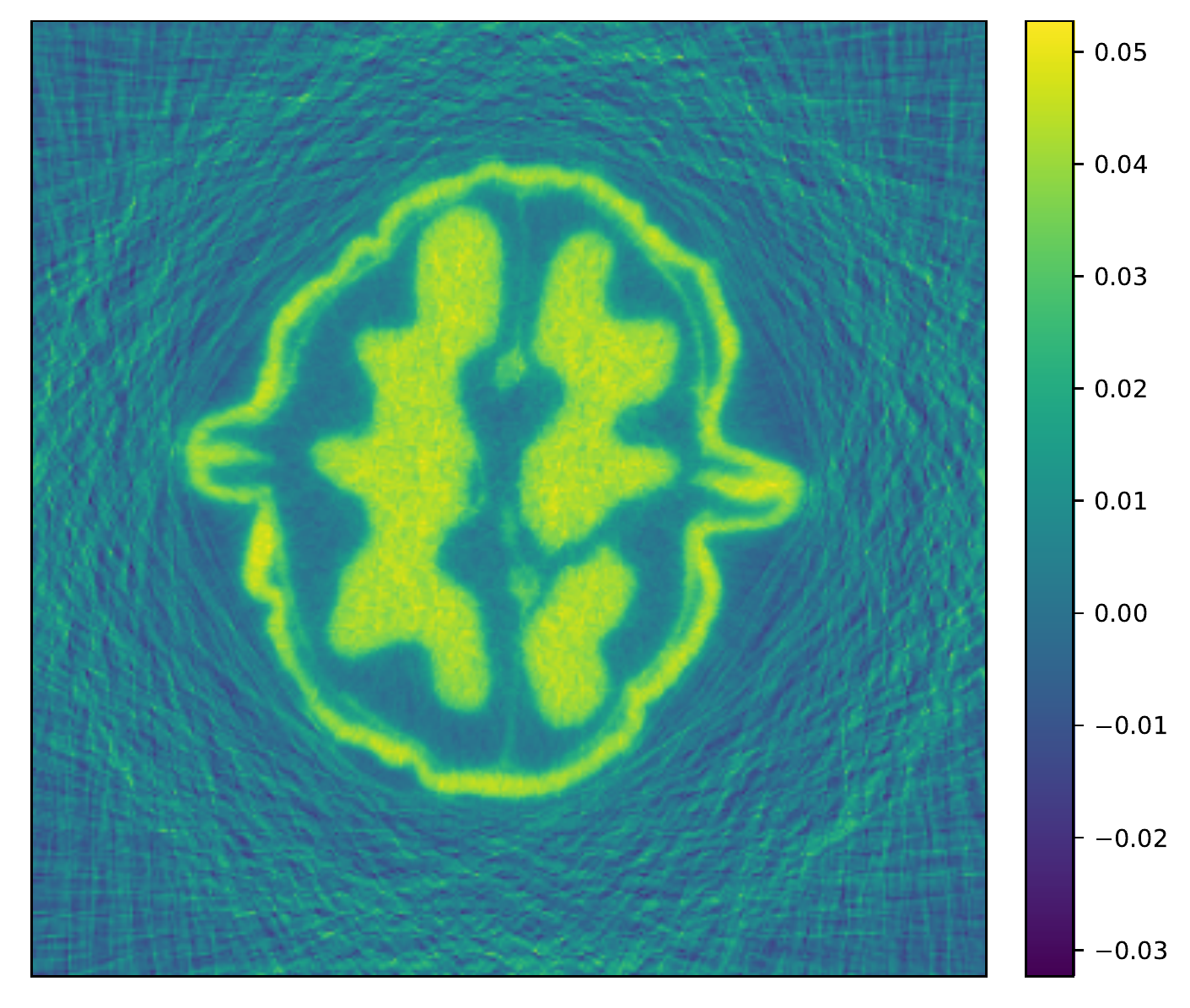} 
};%
\spy on (-0.4cm,-.65cm) in node at (0.75cm, -.55cm);
\end{tikzpicture}\hspace*{-3mm}
&
\begin{tikzpicture}[spy using outlines={rectangle, magnification=3, width = 0.75cm, height = 0.75cm,
connect spies, red, thick}]%
\draw (0,0) node [anchor=north] {\includegraphics[height=0.24\textwidth]{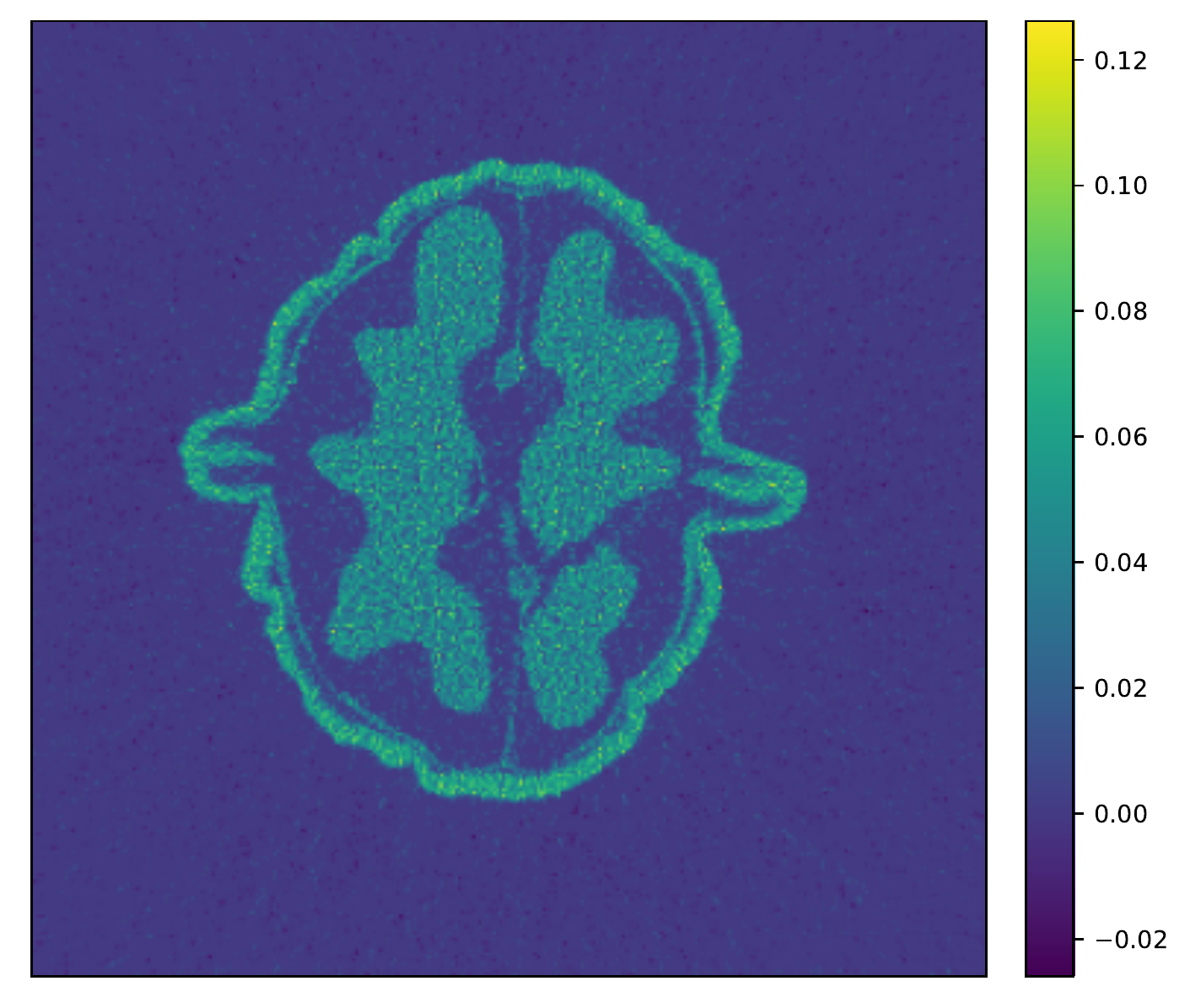} 
};%
\spy on (-0.4cm,-.66cm) in node at (0.75cm, -.56cm);
\end{tikzpicture}\hspace*{-3mm} &
\begin{tikzpicture}[spy using outlines={rectangle, magnification=3, width = 0.75cm, height = 0.75cm,
connect spies, red, thick}]%
\draw (0,0) node [anchor=north] {\includegraphics[height=0.24\textwidth]{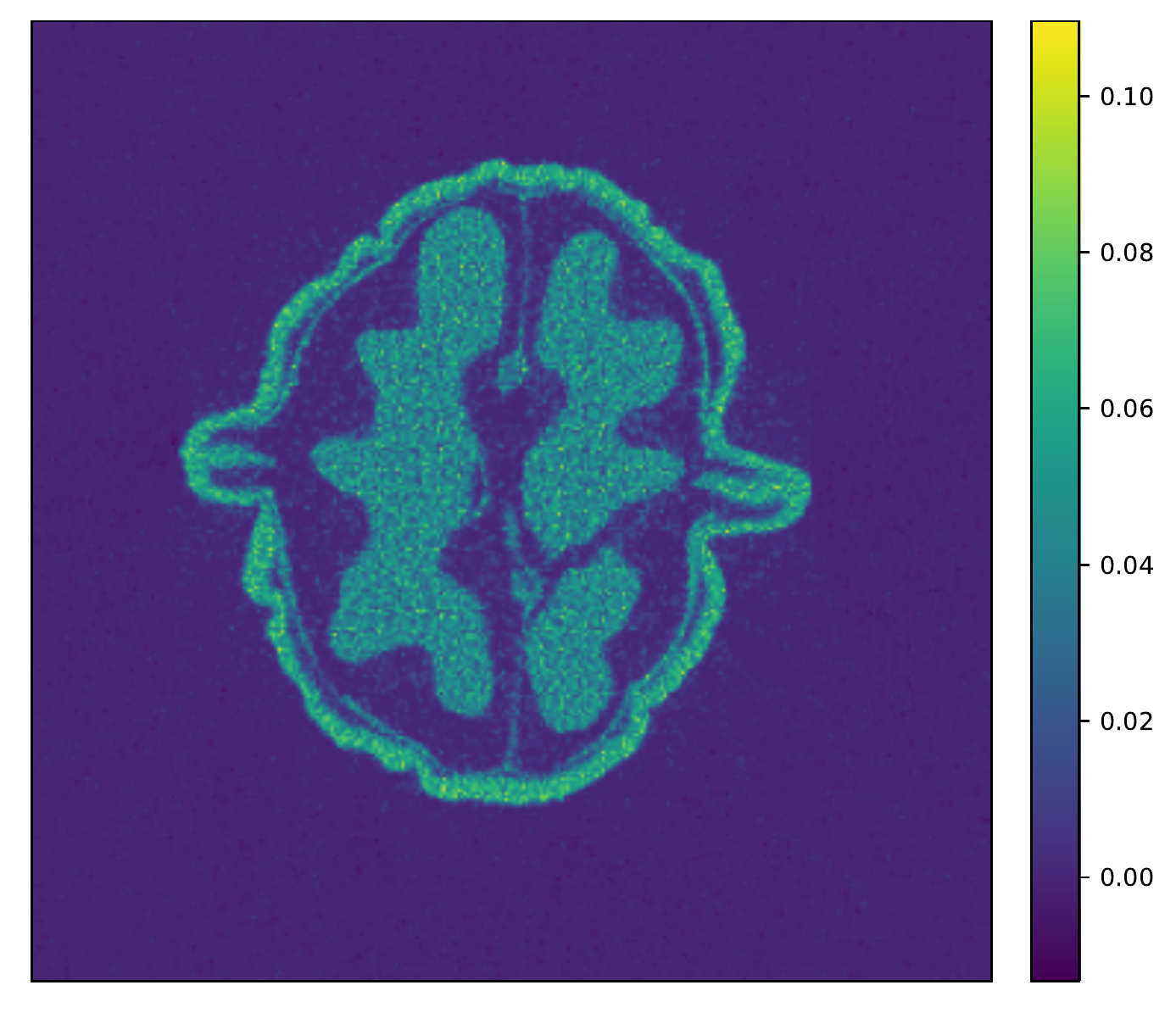}  
};%
\spy on (-0.37cm,-.66cm) in node at (0.78cm, -.56cm);
\end{tikzpicture}\hspace*{-3mm}
\\
{\scriptsize{(a) Sinogram}} & {\scriptsize{(b) SGD}} & {\scriptsize{(c) Constant exponents}} & {\scriptsize{(d) Variable exponents}}
\end{tabular}
\caption{\footnotesize{(a) Noisy sinogram with $10\%$ salt \& pepper (background) and speckle noise with $0$ mean and variance $0.01$ (foreground). (b) $\mathbf{SGD_2}$ result. (c) $\mathbf{SGD_{p_n,1.1}}$ result (d) $\mathbf{SGD_{p_n,q_n}}$ result. $p_-=1.2$, $p_+=1.3$, $q_{-}=1.1$ and $q_{+}=1.9$.}}
\label{fig:walnut_snpbackspecklefore_hilbert}
\vspace{-0.25cm}
\end{figure}
We consider a more delicate noise setting that requires exponential maps which vary in the acquisition domain.
Here, we assume that noise has a different effect on the background (zero entries) and the foreground (non-zero entries) of the clean sinogram.
Namely, we apply $10\%$ salt and pepper noise to the background, and speckle noise with mean $0$ and variance $0.01$ to the foreground, cf. Fig. \ref{fig:walnut_snpbackspecklefore_hilbert}(a) for the resulting noisy sinogram.
Notably, since this noise model has a non-uniform effect across the measurement data, Banach space methods favouring the adjustment of the Lebesgue exponents are expected to perform better than those making use of a constant value. 
Taking as a reference the result obtained by $\mathbf{SGD_2}$ (Fig. \ref{fig:walnut_snpbackspecklefore_hilbert}(b)), we compare here the effect of allowing variable exponents in the solution space only with the effect of allowing both maps $(p_n)$ and $(q_n)$ to be chosen.
By choosing $(p_n)$ based on the initial image and interpolating it between $p_{-}=1.2$ and $p_{+}=1.3$ we then compare $\mathbf{SGD_{p_n,1.1}}$ (i.e., fixed exponent $q=1.1$ in the measurement space), cf.~ Fig. \ref{fig:walnut_snpbackspecklefore_hilbert}(c), with $\mathbf{SGD_{p_n,q_n}}$ where $(p_n)$ is as before while $(q_n)$ is chosen from the sinogram by interpolating between $q_{-}=1.1$ and $q_{+}=1.9$, cf.~Fig. \ref{fig:walnut_snpbackspecklefore_hilbert}(d).
The results show that a flexible framework where both maps $(p_n)$ and $(q_n)$ adapt to local contents are more suited for dealing with this challenging scenario.